\documentclass[11pt]{amsart}

\usepackage[marginratio=1:1]{geometry}
\usepackage{color}
\definecolor{refs}{rgb}{0.7,0,0}
\definecolor{ext}{RGB}{112,112,112}
\definecolor{cite}{RGB}{034,113,179}
\usepackage[colorlinks=true,citecolor=cite,linkcolor=refs,urlcolor=ext,backref=page]{hyperref}

\newtheorem{theorem}{Theorem}[section]

\newtheorem{proposition}[theorem]{Proposition}

\theoremstyle{definition}

\newtheorem*{remark}{Remark}

\newcommand{\R}{\mathbb{R}}
\newcommand{\N}{\mathbb{N}}

\newcommand{\II}{\mathcal{I}}

\newcommand{\MM}{\mathcal{M}}

\title{The Schwarzian derivative and Euler--Lagrange equations}
\author{Wojciech Kry\'nski}
\thanks{Partially supported by the grant 2019/34/E/ST1/00188 from the National Science Centre, Poland.}
\address{
Institute of Mathematics, Polish Academy of Sciences, ul. \'Sniadeckich 8, 00-656 Warszawa, Poland}

\email{krynski@impan.pl}


\begin{document}

\begin{abstract}
We study the Schwarzian derivative from a variational viewpoint. Firstly we show that the Schwarzian derivative defines a first integral of the Euler--Lagrange equation of a second order Lagrangian. Secondly, we show that the Schwarzian derivative itself is the Euler--Lagrange operator for an appropriately chosen class of variations.
\end{abstract}

\maketitle

\section{Introduction}
The Schwarzian derivate of function $u$ is defined as
\[
S(u)=\frac{u'''}{u'}-\frac{3}{2}\left(\frac{u''}{u'}\right)^2.
\]
The derivative appears in the projective and conformal geometry as well as in many other contexts in mathematics and mathematical physics (see e.g. \cite{AW,BE,G,Gi,OS} and \cite{OT} for a survey). In this note we intend to exhibit its variational nature.  For this we shall consider the following equation
\begin{equation}\label{eq_schwarz}
S(u)=0.
\end{equation}
Interestingly, the equation turns out to be a one-dimensional counterpart of the conformal geodesic equation known in the conformal geometry \cite{BE, CG, DK, FH, FS, SZ}. We show that the variational approach to the conformal geodesics recently proposed in \cite{DK} can be adapted in the context of the Schwarzian derivative as well. Consequently we interpret \eqref{eq_schwarz} as an Euler--Lagrange equation. Note that a variant of the Schwarzian mechanics has been studied in \cite{G} where a Hamiltonian approach is developed.

Our results split into two parts. Firstly, we define a second order Lagrangian function which gives rise to a 4th order Euler--Lagrange equation with $S(u)$ being a first integral. In particular, solutions to \eqref{eq_schwarz} form a special subclass of solutions of the 4th order equation. Further, we study the 4th order equation from the geometrical viewpoint and characterize the special class of solutions corresponding to $S(u)=0$ in terms of the W\"unschmann invariants \cite{B,DT}. It turns out that the solution space of the equation is equipped with a field of curves in the projective tangent bundle that reduce to the light cones of a flat Lorentzian conformal metric on a submanifold corresponding to \eqref{eq_schwarz} (this is consistent with \cite{T} where the geometric structure of \eqref{eq_schwarz} is considered). On the other hand, in neighborhoods of points of the solution space where $S(u)$ does not vanish we get new examples of the so-called isotrivial causal geometry (see \cite{M} and \cite{KM}).
 
Secondly, we aim to express equation \eqref{eq_schwarz} itself as the Euler--Lagrange equation. A priori, one expects the Euler--Lagrange equations to be of even order. However, once the class of variations with respect to which the critical points are calculated is extended, it becomes harder to be a critical point and consequently the Euler--Lagrange equations may be of lower order, not necessarily even. This approach has been successfully developed in \cite{DK}. We show that it can be adopted in the case of equation \eqref{eq_schwarz}, too.

\section{Variational approach to the Schwarzian derivative}

\subsection{Notation}
Recall that a variation of a given function $u\colon[t_0,t_1]\to\R$ is a mapping $\varphi\colon[t_0,t_1]\times[-1,1]\to\R$ such that $\varphi(t,0)=u(t)$.  Denote $v(t)=\frac{\partial}{\partial s}\varphi(t,s)|_{s=0}$ and call it a variational vector field along $u$. In the standard approach to the variational problems one often assumes that  $\varphi(t_i,s)$ coincides with $u(t_i)$ up to certain order which translates to the fact that $v(t_i)=v'(t_i)=\ldots=v^{(k)}(t_i)= 0$, for an appropriate $k\in\N$ and $i=0,1$. However, in due course we shall provide a more subtle class of boundary conditions that will result in a more sophisticated class of variations that will be better adapted to the purposes of our approach. 

Let $\II$ be a given integral functional. The variation of $\delta_v\II[u]$ with respect to a variational vector field $v$ is defined as 
\[
\delta_v\II[u]=\frac{d}{ds}\II[\varphi(t,s)]|_{s=0}.
\]
We shall say that a function $u$ is a critical point of $\II$ with respect to a class of admissible variations if $\delta_v\II_F[u]=0$ for all $v$'s that are considered.

\subsection{4th order Euler--Lagrange equations}
In this section we shall consider a second order Lagrangian in the form
\begin{equation}\label{eq_L}
L(u,u',u'')=\left(\frac{u''}{u'}\right)^2
\end{equation}
and the associated integral functional
\begin{equation}
\II_L[u]=\int_{t_0}^{t_1}L(u,u',u'') dt,
\end{equation}
where $[t_0,t_1]$ is a fixed interval in $\R$. Note that this functional is a one-dimensional counterpart of the functional introduced by T.\,Bailey and M.\,Eastwood in \cite{BE} in the context of the conformal geometry.  Indeed, setting $n=1$ in \cite{BE} reduces the Lagrangian of \cite{BE} in the flat case to \eqref{eq_L}. However, there is a substantial difference in comparison to the conformal geometry. Namely, one cannot prove conformal invariance of \eqref{eq_L} as there is no Schouten tensor in dimension~1 which adds an extra term in \cite{BE}. On the other hand, Lagrangian \eqref{eq_L} is invariant with respect to the affine transformations. 

We start with the classical problem of finding critical points of $\II_L$ subject to the class of variations that keep the endpoints fixed up to the first order.
\begin{theorem}
The Euler--Lagrange equation associated to \eqref{eq_L} takes the form
\begin{equation}\label{eq_EL4}
\frac{u''''}{u'}=-3\frac{(u'')^3}{(u')^3}+4\frac{u'''u''}{(u')^2}.
\end{equation}
Moreover, any solution to the equation $S(f)=0$ satisfies \eqref{eq_EL4}.
\end{theorem}
\begin{proof}
While it is straightforward to derive the Euler--Lagrange equation for $\II_L$ we shall look closer at calculations in order to prove the second statement of the theorem and also for the purposes of the subsequent parts of the paper. 
We compute the derivative $\frac{d}{ds}$ under the integral and then integrate by parts appropriate terms. As a result, $\delta_v\II_L[u]$ can be written in one of the following forms
\begin{eqnarray}
\delta_v\II_L[u]&=&\int_{t_0}^{t_1}\left(2\frac{u''v''}{(u')^2}-2\frac{(u'')^2v'}{(u')^3}\right)dt\nonumber\\
&=&\int_{t_0}^{t_1}\left(-2\frac{u'''}{(u')^2}+2\frac{(u'')^2}{(u')^3}\right)v'dt+B_0|_{t_0}^{t_1}\label{eq_var3a}\\
&=&\int_{t_0}^{t_1}\left(-2\frac{u'''}{u'}+3\frac{(u'')^2}{(u')^2}\right)((u')^{-1}v'-u''(u')^{-2}v)dt+B_1|_{t_0}^{t_1}\label{eq_var3}\\
&=&\int_{t_0}^{t_1}\left(2\frac{u''''}{u'}+6\frac{(u'')^3}{(u')^3}-8\frac{u'''u''}{(u')^2}\right)(u')^{-1}vdt+B_2|_{t_0}^{t_1}\label{eq_var4}
\end{eqnarray}
where the boundary terms are respectively
\[
B_0=\frac{2u''v'}{(u')^2},\qquad B_1=\frac{2u''v'}{(u')^2}-\frac{(u'')^2 v}{(u')^3},\qquad B_2=\frac{2u''v'}{(u')^2}-\frac{2u'''v}{(u')^2}+\frac{2(u'')^2 v}{(u')^3}.
\]
Under the assumption that $\varphi(t_0,s)=u(t_0)$ and $\varphi(t_1,s)=u(t_1)$ up to the first order, i.e. $v(t_i)=0$ and $v'(t_i)=0$ for $i=0,1$, the boundary terms vanish. Consequently the fundamental lemma of the calculus of variations applied to \eqref{eq_var4} implies that \eqref{eq_EL4} has to be satisfied.
Note that the intermediate step \eqref{eq_var3} is not necessary for the derivation of \eqref{eq_EL4}. However, it is obtained from \eqref{eq_var3a}, by a simple integration by parts and implies that $\delta_v\II_L[u]=0$ for $u$ satisfying \eqref{eq_schwarz}. Consequently \eqref{eq_EL4} holds provided that \eqref{eq_schwarz} holds.
\end{proof}

Closer look at the calculations implies the following
\begin{proposition}\label{prop_integrals}
$S(u)$ as well as
\[
C(u)=(u')^{-1}\left(\frac{u'''}{u'}-\frac{(u'')^2}{(u')^2}\right)
\] 
are first integrals of \eqref{eq_EL4}.
\end{proposition}
Note that a multidimensional counterpart of $C(u)$ appeared in the upper mentioned context of the conformal geodesics where $\frac{d}{dt}C=0$ is referred to as the Mercator equation \cite{DK}. We shall exploit the fact that $S(u)$ is the first integral to get explicit solutions to \eqref{eq_EL4}.

\begin{proposition}
A general solution solution to \eqref{eq_EL4} is 
\begin{equation}\label{eq_solution}
u(t)=
\begin{cases}
\frac{Ae^{2\sqrt{c}t}+B}{Ce^{2\sqrt{c}t}+D}&\text{if } S(u)=-c<0\\ 
\frac{At+B}{Ct+D}&\text{if } S(u)=0\\
\frac{A\tan\left(\sqrt{c}t\right)+B}{C\tan\left(\sqrt{c}t\right)+D}&\text{if } S(u)=c>0 
\end{cases}
\end{equation}
where $A,B,C,D$ are arbitrary constants such that $AB-CD\neq0$.
\end{proposition}\label{prop_solution}
\begin{proof}
Since $S(u)$ is the first integral of \eqref{eq_EL4} it is sufficient to solve
\[
S(u)=c
\]
for any constant $c$. The solution depends on the sign of $c$. In fact (c.f. \cite{OT}), $u$ can be found as a ratio of two independent solutions to
\[
\psi''=-c\psi
\]
which gives \eqref{eq_solution} as a result.
\end{proof}

\begin{remark}
Solution $u(t)$ for $S(u)<0$ appeared in \cite{G} as a starting point for the studies of Hamiltonian mechanics related to the Schwarzian.
\end{remark}

\subsection{Geometric structure of the 4th order equation}
In this section we shall describe the geometry related to equation \eqref{eq_EL4} and characterize \eqref{eq_schwarz} in terms of the contact invariants of ODEs. For this we write \eqref{eq_EL4} in the form
\begin{equation}\label{eq_EL4a}
u''''=F(u',u'',u'')
\end{equation}
where
\[
F(p,q,r)=-3\frac{r^3}{p^2}+4\frac{qr}{p}.
\]
We are interested in \eqref{eq_EL4a} under the action of the group of contact transformations. There are two basic invariants of 4th order equations introduced by R.~Bryant in \cite{B}. We shall denote them $W_0$ and $W_1$. The two invariants are usually referred to as the generalized W\"unschmann invariants, as they are direct generalization of the W\"unschmann invariant introduced for 3rd order equations (see \cite{D,DT}). The following explicit expressions can be found in \cite{DT}
\[
W_1=\frac{9}{4}F_r\frac{d}{dt}F_r-\frac{3}{2}\frac{d^2}{dt^2}F_r+3\frac{d}{dt}F_q-\frac{3}{8}F_r^3-\frac{3}{2}F_qF_r-3F_p
\]
and
\[
\begin{aligned}
W_0=&\frac{11}{1600}F_r^4-\frac{9}{50}F_r^2\frac{d}{dt}F_r-\frac{1}{200}F_r^2F_q+\frac{21}{100}\left(\frac{d}{dt}F_r\right)^2+\frac{1}{50}\frac{d}{dt}F_rF_q-\frac{9}{100}F_q^2\\
&+\frac{7}{20}F_3\frac{d^2}{dt^2}F_r-\frac{1}{5}\frac{d^3}{dt^3}F_r+\frac{3}{10}\frac{d^2}{dt^2}F_q-\frac{1}{4}F_r\frac{d}{dt}F_q,
\end{aligned}
\]
where subscripts denote partial derivatives and $\frac{d}{dt}=\partial_t+p\partial_u+q\partial_p+r\partial_q+F\partial_r$ is the total derivative. If both invariants vanish then the solution space of an equation is equipped with a $GL(2)$-structure \cite{B,DT}. That is not the case for equation \eqref{eq_EL4a}. In fact, equation \eqref{eq_EL4a} can serve as a very interesting example of an ODE for which only half of the  W\"unschmann invariants vanish. Indeed we have
\begin{theorem}\label{thm2}
The W\"unschmann invariants of equation \eqref{eq_EL4a} satisfy $W_1=0$ and
\[
W_0=-\frac{36}{100}S(u)^2.
\]
The projective tangent bundle $P(T\MM)$ of the solution space $\MM$ of \eqref{eq_EL4a} is equipped with a field of curves that are
\begin{itemize}
\item[(a)] projectively equivalent to the curve
\[
t\mapsto(e^{-t},t,-e^t)
\]
at points of $\MM$ corresponding to solutions such that $S(u)<0$,
\item[(b)] projectively equivalent to the curve
\[
t\mapsto(t-\sin t,\cos t,t+\sin t)
\]
at points of $\MM$ corresponding to solutions such that $S(u)>0$, 
\item[(c)] projectively equivalent to the rational normal curve at points of $\MM$ corresponding to solutions such that $S(u)=0$.
\end{itemize}
\end{theorem}
\begin{proof}
In order to find $W_0$ and $W_1$ one uses the explicit formulae provided above. It is a matter of computations to verify the assertion $W_1=0$ and $W_0=-\frac{36}{100}S(u)^2$.

In order to prove the second part of the theorem we recall that the solution space of a 4th order ODE can be defined as $\MM=J^3(\R,\R)/X_F$ where $J^3(\R,\R)$ is the space of 3-jets of functions $\R\to\R$ and $X_F$ is the total derivative vector field $X_F=\frac{d}{dt}=\partial_t+p\partial_u+q\partial_p+r\partial_q+F\partial_r$ where $(t,x,p,q,r)$ are standard coordinates on $J^3(\R,\R)$: $t$ and $u$ are the independent and dependent variables, $p=u'$, $q=u''$ and $r=u'''$. There is a canonical quotient map $\pi\colon J^3(\R,\R)\to\MM$ and the field of curves in $P(T\MM)$ is geometrically defined as the projectivization of the union of all 1-dimensional subspaces in $T\MM$ of the form $\pi_*(\mathrm{span}(\partial_q))$. Indeed, $\pi_*(\mathrm{span}(\partial_q))$ defines  a one-parameter family of lines in $T_s\MM$ for any solutionon $s\in\MM$. After the projectivization the family of lines becomes a curve in $P(T_s\MM)$.

Using \eqref{eq_solution}, and applying the following point transformation of $J^0(\R,\R)$
\[
(t,u)\mapsto \left(\frac{1}{\tilde c}t,\frac{Du-B}{-Cu+A}\right)
\]
for an appropriate choice of $A,B,C,D$ and certain $\tilde c>0$, one gets that a general solution to \eqref{eq_EL4a} reduces to either $u(t)=e^{t}$, $u(t)=t$, or $u(t)=\tan(t)$, depending on the original sign of $S(u)$. It follows that it is sufficient to find curves in  $P(T\MM)$ at points corresponding to the three particular solutions.

We exploit the approach of \cite{D} that uses linearization of an ODE. Indeed, as proved in \cite{D} the W\"unschmann invariants of an ODE coincides with the Wilczynski invariants of the linearized equation, and in order to find the upper mentioned curves it is sufficient to consider the linearized equation. The linearization at $u(t)=e^t$ takes the form
\[
v''''=4v'''-5v''+2v'
\]
with a general solution $v(t)=\alpha e^t+\beta te^t+\gamma e^{2t}+\delta$. The linearization at $u(t)=\tan(t)$ takes the form
\[
v''''=8\tan(t)v'''+(8-12\tan^2(t))v''-16\tan(t)v'
\]
with a general solution $v(t)=\alpha \tan(t)+\beta\frac{t}{\cos^2(t)}+\gamma\tan^2(t)+\delta$. Finally, the linearization at $u(t)=t$ is $v''''=0$ with a general solution $v(t)=\alpha t^3+\beta t^2+\gamma t+\delta$. In all three cases the constants $(\alpha,\beta,\gamma, \delta)$ are coordinates on the respective solution space $\MM$ of the linearized equation and the aforementioned curves in $P(T\MM)$ turn out to be the dual projective curves to the curves $t\mapsto (e^t:te^t: e^{2t}:1)$, $t\mapsto(\tan(t):\frac{t}{\cos^2(t)}:\tan^2(t):1)$ or $t\mapsto(t^3:t^2:t:1)$, respectively.
\end{proof}
\begin{remark}
Note that the geometry of the Schwarzian equation \eqref{eq_schwarz} is investigated in \cite{T} as an example of a 3rd order ODE. It is proved that the W\"unschmann invariant for the equation $S(u)=0$ vanishes and consequently the solution space has a natural conformal structure of Lorentzian signature (which is a 3-dimensional counterpart of $GL(2)$-structures). That is consistent with Theorem \ref{thm2}, where both W\"unschmann invariants vanish for special solutions satisfying \eqref{eq_schwarz}.

For $S(u)\neq 0$ we get two open subsets of the solution space equipped with fields of certain curves in the projective tangent bundles. The osculating cones of the curves define surfaces in the projective bundles, which can be referred to as the causal structures, as they replace the light cones of a conformal metric and can define (at least locally) a casualty relation between points. We refer to \cite{M} (see also  \cite{KM}) for a comprehensive study of the causal structures.
\end{remark}

\subsection{Extended variations}
Note that Lagrangian \eqref{eq_L} differs from the Schwarzian $S(u)$ only by a differential. Indeed one has
\begin{equation}\label{eq_bb}
S(u)=\frac{d}{dt}\left(\frac{u''}{u'}\right)-\frac{1}{2}L.
\end{equation}
It follows that the Euler--Lagrange equation for the integral functional
\[
\II_S=\int_{t_0}^{t_1}S(u)dt
\]
coincide with the one for $\II_L$ (note that although $S(u)$ is of third order, the associated Euler--Lagrange equation is of 4th order because $S(u)$ is linear in $u'''$). The functional $\II_S$ can be viewed as a reduction of a functional introduced in \cite{DK} for the conformal geodesics. Note that the third order Lagrangian of \cite{DK} has a natural interpretation in the tractor calculus (c.f. \cite{CG,DK,GST,SZ}). 

We shall now exploit $\II_S$ and show that one can actually get equation \eqref{eq_schwarz} as the Euler--Lagrange equation. Indeed, as announced in the Introduction, we shall apply an idea developed in \cite{DK} of appropriate choice of admissible variations. Namely, we enlarge the class of variations such that they no longer keep endpoints of a curve fixed.  However, in order to apply this approach one needs to adjust the fundamental lemma of the calculus of variations properly, and take care of the boundary terms. That is the reason we shall use $\mathcal{I}_S$ instead of $\mathcal{I}_L$. Indeed, as we shall see $\mathcal{I}_S$ leads to second-order boundary conditions which will give necessary freedom to apply a variant of the fundamental calculus of variations.

We shall also use the following first order differential operator acting on variational vector fields $v$ along given $u$
\[
D_u(v)=v'-u''(u')^{-1}v.
\]
This operator appeared in formula \eqref{eq_var3} before and will be of interests in the following
\begin{theorem}\label{thm3}
Function $u$ is a solution to equation \eqref{eq_schwarz}  if and only if $u$ is a critical point of $\II_S$ with respect to the class of variations $v$ such that the endpoint condition
\[
(D_u^2(v)+S(u)v)|^{t_1}_{t_0}=0
\]
is satisfied.
\end{theorem}
\begin{proof}
By \eqref{eq_bb} we have $\II_S=\left.\left(\frac{u''}{u'}\right)\right|^{t_1}_{t_0}-\frac{1}{2}\II_L$. Using \eqref{eq_var3} for $\II_L$ we find that
\begin{equation}\label{eq_aa}
\delta_v\II_S[u]=\int_{t_0}^{t_1}S(u)D_u(v)dt + B|^{t_1}_{t_0}
\end{equation}
where the boundary term $B$, which follows from \eqref{eq_bb}, equals
\[
B=\frac{v''}{u'}-\frac{u''v'}{(u')^2}-\frac{1}{2}B_1=\frac{v''}{u'}-2\frac{u''v'}{(u')^2}+\frac{(u'')^2 v}{2(u')^3},
\]
and can be equivalently written as
\[
B=(u')^{-1}(D_u^2(v)+S(u)v).
\]

Now, we consider a class of variations such that $B|^{t_1}_{t_0}=0$. Thus, \eqref{eq_aa} takes the form
\begin{equation}\label{eq_cc}
\delta_v\II_S[u]=\int_{t_0}^{t_1}S(u)D_u(v)dt
\end{equation}
and we would like to deduce that the assumption that $\delta_v\II_S[u]=0$ for all $v$ satisfying $B|^{t_1}_{t_0}=0$ implies that $S(u)=0$. To prove this it would be enough to take as $D_u(v)$ in \eqref{eq_cc} an arbitrary bump function $\varphi$ concentrated in a neighborhood of a point $t\in(t_0,t_1)$ (compare the classical proof of the fundamental lemma of the calculus of variations). This, in general, is not possible for $v$ satisfying $B|^{t_1}_{t_0}=0$. Nevertheless, we can always solve the first order ODE: $D_u(v)=\varphi$ for $v$. If $B|^{t_1}_{t_0}=0$ holds for this particular $v$ we are done. Otherwise, we show that there is a perturbation $\tilde v=v+\hat v$ such that $\hat v$ that is arbitrarily small in $H^1$-norm and such that $B|^{t_1}_{t_0}=0$ is satisfied for $\tilde v$. For such $\tilde v$, $D(\tilde v)$ will be arbitrarily close to the bump function $\varphi$ and that will be sufficient to complete the reasoning.

The perturbation $\hat v$ can be found using the fact that the condition $B|^{t_1}_{t_0}=0$ depends on second derivative of $v$ in neighborhoods of endpoints $t_0$ and $t_1$. Indeed, as $\hat v$ we take a segment of a parabola in a small neighborhood of $t_0$ smoothly glued with a constant zero function on the rest of the interval $[t_0,t_1]$. The parabola can be taken such that $\hat v$ and $D_u(\hat v)$ are arbitrary small (indeed, take: $t\mapsto c(t-t_0-\epsilon)^2$ for a small $\epsilon>0$ and some constant $c$). On the other hand $D^2_u(\hat v)(t_0)$ can be arbitrary large (as it depends on the value of $c$) and can be adjust such that $v+\hat v$ satisfies the boundary condition $B|^{t_1}_{t_0}=0$ (see also Corollary 4.2 of \cite{DK} for a similar construction).
\end{proof}

\end{document}